\scrollmode
\documentclass[10pt]{amsart}
\usepackage{amssymb}
\usepackage[all,arc]{xy}
\usepackage{multicol}

\theoremstyle{plain}
\newtheorem{prop}{Proposition}
\newtheorem{thm}[prop]{Theorem}
\newtheorem*{thmA}{Theorem A}
\newtheorem*{corB}{Corollary B}
\newtheorem*{thmC}{Theorem C}
\newtheorem{cor}[prop]{Corollary}
\newtheorem{Lemma}[prop]{Lemma}
\newtheorem{fact}[prop]{Fact}

\newtheorem*{quesA}{Question*}

\theoremstyle{definition}

\theoremstyle{remark}
\newtheorem{rem}[prop]{Remark}

\numberwithin{prop}{section}
\numberwithin{example}{section} 
\numberwithin{equation}{section}

\newcommand{\Z}{\mathbb{Z}}
\newcommand{\Q}{\mathbb{Q}}
\newcommand{\N}{\mathbb{N}}
\newcommand{\F}{\mathbb{F}}

\newcommand{\rk}{\mathrm{rk}}
\newcommand{\tG}{\tilde{G}}
\newcommand{\ab}{\mathrm{ab}}
\newcommand{\HNN}{\mathrm{HNN}}
\newcommand{\het}{\mathrm{ht}}

\newcommand{\tg}{\tilde{g}}
\newcommand{\stab}{\mathrm{stab}}
\newcommand{\caR}{\mathcal{R}}
\newcommand{\IFP}{\textbf{IF}}

\newcommand{\caT}{\mathcal{T}}
\newcommand{\gG}{\mathcal{G}}
\newcommand{\spn}{\mathrm{span}}
\newcommand{\dbs}{\backslash\!\!\backslash}
\newcommand{\tLambda}{\widetilde{\Lambda}}
\newcommand{\st}{\mathfrak{st}}
\newcommand{\be}{\mathbf{e}}
\newcommand{\beb}{\bar{\be}}
\newcommand{\bof}{\mathbf{f}}
\newcommand{\tv}{\tilde{v}}
\newcommand{\tw}{\tilde{w}}

\newcommand{\image}{\mathrm{im}}
\newcommand{\kernel}{\mathrm{ker}}
\newcommand{\Tor}{\mathrm{Tor}}
\newcommand{\tr}{\mathrm{tr}}

\newcommand{\argu}{\hbox to 1.5ex{\hrulefill}}  % placeholder

\begin{document}

\title{Normal subgroups of limit groups of prime index}
\author{Jhoel S. Gutierrez}

\author{Thomas S. Weigel}

\thanks{The first author was supported by CNPq-Brazil}

\address{ 
Dipartimento di Matematica e Applicazioni\\
Universit\`a degli Studi di Milano-Bicocca\\
Ed.~U5, Via R.Cozzi 55\\
20125 Milano, Italy}
\email{thomas.weigel@unimib.it}
\email{jhoelsg31@gmail.com}

\date{\today}

\begin{abstract} 
Motivated by their study of pro-$p$ limit groups,
D.H. Koch\-loukova and P.A.~Zalesskii formulated in \cite[Remark after Thm.~3.3]{pav:pplim}
a question concerning the minimum number of generators $d(N)$ of a normal subgroup $N$ 
of index $p$ in a non-abelian limit group $G$ (cf. Question*). It is shown that  
the analogous question for the rational rank has an affirmative answer (cf. Thm.~A).
From this result one may conclude that the original question of 
D.H.~Koch\-loukova and P.A.~Zalesskii
has an affirmative answer
if the abelianization $G^{\ab}$ of $G$ is torsion free and $d(G)=d(G^{\ab})$ (cf. Cor.~B),
or if $G$ has the \IFP-property (cf. Thm~C).
\end{abstract}

\maketitle
%%%%%%%%%%
%%%July 7th, 2016
%%%%%%%%%%
\section{Introduction}
\label{s:intro}
In recent years {\em limit groups} (or {\em $\omega$-residually free groups})
have received much attention (cf. \cite{AB}, \cite{BF}, \cite{CG},
\cite{Guirardel}, \cite{Paulin})
primarily due to the groundbreaking work of Z.~Sela (cf. \cite{S1}, \cite{S2} and the references therein).
This class of groups was first introduced
by B.~Baumslag in \cite{Ba}  under the more traditional name of 
{\em{fully residually free groups}}, and subsequently studied intensively
by many authors (cf.  \cite{B}, \cite{Kapovich}, \cite{KMRS}, \cite{KM1}, \cite{KM2}). 
Indeed, this notion reflects the fact that for any limit group $G$
and any finite subset $T$ of $G$ there exists a homomorphism
from $G$ to a free group $F$ that is injective on $T$.
Examples of limit groups include all finitely generated free groups,
all finitely generated free abelian groups, and all the fundamental groups
of closed oriented surfaces. Moreover, the class of limit groups is closed
with respect to finitely generated subgroups and free products.
This fact can be used to construct many examples. 
More sophisticated examples can be found in some of the articles cited above.

The purpose of this note is to investigate the following question
which was raised by D.H. Koch\-loukova and P.A.~Zalesskii in \cite{pav:pplim}, where they 
answered the analogous question for pro-$p$ limit groups.

\begin{quesA}
Let  $G$ be a  non-abelian limit group, and let $U$ be a  normal subgroup  of $G$ of  prime index $p$. 
Does this imply that  $d(U)> d(G)$?
\end{quesA}

Here $d(G)$ denotes the minimum number of generators of a finitely generated group $G$.
For a pro-$p$ group $\tG$ the minimum number of (topological) generators $d(\tG)$ of $\tG$
is closely related to the cohomology
group $H^1(\tG,\F_p)$, where $\F_p$ denotes the finite field with $p$ elements,
i.e., $d(\tG)=\dim_{\F_p}(H^1(\tG,\F_p))$ (cf. \cite[\S I.4.2, Cor. of Prop.~25]{ser:gal}).
For an abstract group $G$ such a close relation does not hold.
There are two homological invariants of a finitely generated group $G$
which can be seen as a homological approximation of $d(G)$: The {\em rational rank} of $G$ given by
\begin{equation}
\label{eq:defQ}
\rk_{\Q}(G)=\dim_{\Q}(G^{\ab}\otimes_{\Z}\Q)=\dim_{\Q}(H_1(G,\Q)),
\end{equation}
where $G^{\ab}=G/[G,G]$ denotes the {\it abelianization} of $G$, and
\begin{equation}
\label{eq:dGab}
d(G^{\ab})=\rk_{\Z}(G^{\ab})=\rk_{\Z}(H_1(G,\Z)).
\end{equation}
In particular, $\rk_{\Q}(G)\leq d(G^{\ab})\leq d(G)$. 
The main result of this paper is an affirmative answer to the analogue of Question* for the rational rank
(cf. Thm.~\ref{thm:A}).

\begin{thmA}
Let  $G$ be a  non-abelian limit group,  and let $U$ be a  normal subgroup  of  $G$ of prime index $p$. 
Then $\rk_{\Q}(U)> \rk_{\Q}(G)$.
\end{thmA}

From Theorem~A one concludes the following 
affirmative partial answer to Question* (cf. Corollary~\ref{cor1}).

\begin{corB}
Let $G$ be a non-abelian limit group such that $G^{ab}$ is torsion free group and that $d(G)=d(G^{ab})$. If $U$ is a normal subgroup of $G$ of prime index $p$,  then $d(U)>d(G)$.
\end{corB}

It should be mentioned that there exist limit groups
$G$ for which $d(G^{\ab})\not=d(G)$ (cf. Remark~\ref{rem:notab}),
and also for which $G^{\ab}$ is not torsion free (cf. Remark~\ref{rem:tflim}),
and one expects that Question* has an affirmative answer for all limit groups.
A group $G$ is said to have the {\em \IFP-property}, if every subgroup $U$
which is of infinite index in $G$ is free.
For the class of limit groups with the \IFP-property the proof of Theorem~A 
can be modified in order to deduce the following (cf. Thm.~\ref{thm:1rel}).
 
\begin{thmC}
Let $G$ be a limit group with the \textbf{IF}-property, 
and let $U$ be a normal subgroup of $G$ of prime index $p$. Then $d(U)> d(G)$. 
\end{thmC}

%%%%%%%%%%%%%%
%%%% Limit groups %%%
%%%%%%%%%%%%%
%%% July 7th, 2016 %%%%%
\section{Limit groups}
\label{lgt}
By $\N$ we will denote the set of positive integers, and by $\N_0$ the set of
non-negative integers, i.e., $\N_0=\{0,1,2,\ldots\}$. We will use the symbol ``$\star$''
for {\it free products with amalgamation} in the category of groups.

\subsection{Extension of centralizers}
\label{ss:et}
Starting from a limit group $G$ there is a standard procedure
to construct a limit group $G(C,m)$, where $C\subseteq G$ is a maximal cyclic subgroup
of $G$ and $m\in\N$. This procedure is known as {\em extension of centralizers}, i.e., if $G$ is a limit group, then
\begin{equation}
\label{eq:excen}
G(C,m)=G\star_C(C\times\Z^m)
\end{equation}
is again a limit group (cf. \cite[Lemma 2 and Thm.~4]{KM2}).
For short we call a limit group $G$ to be an {\em iterated extension of centralizers of a free group} (= i.e.c.f. group),
if there exists a sequence of limit groups $(G_k)_{0\leq k\leq n}$ such that
\begin{itemize}
\item[(EZ$_1$)] $G_0=F$ is a finitely generated free group, and $G_n\simeq G$;
\item[(EZ$_2$)] for $k\in\{0,\ldots,n-1\}$ there exists a maximal cyclic subsgroup $C_k\subseteq G_k$ and
$m_k\in\N$ such that $G_{k+1}\simeq G_k(C_k,m_k)$.  
\end{itemize}
If $G$ is an iterated extension of centralizers of a free group, one calls the minimum number $n\in\N_0$ for which there exists
a sequence of limit groups $(G_k)_{0\leq k\leq n}$ satisfying (EZ$_1$) and (EZ$_2$) the {\em level} of $G$.
This number will be denoted it by $\ell(G)$. E.g., a finitely generated free group is an iterated extension of 
centralizers of a free group of level $0$, and a finitely generated free abelian group is an
iterated extension of centralizers
of a free group of level $1$.

%%%%%%%
\subsection{The height of a limit group}
\label{ss:hgt}
By the second embedding theorem (cf. \cite[\S 2.3, Thm.~2]{KMRS}), every limit group $G$ is isomorphic to a subgroup
of an i.e.c.f. group $H$. The {\em height} $\het(G)$ of $G$ is defined as
\begin{equation}
\label{eq:defht}
\het(G)=\min\{\,\ell(H)\mid G\subseteq H,\ \text{$H$ an i.e.c.f. group}\,\}.
\end{equation}
E.g., a limit group $G$ is of height $0$ if, and only if, it is a free group of finite rank,
and non-cyclic finitely generated free abelian groups are of height $1$.

%%%%%
\subsection{Limit groups as fundamental groups of graph of groups}
\label{ss:lgrgr}
Let $G$ be a limit group of height $\het(G)=n\geq 1$, and 
let $H$ be an i.e.c.f. group of level $\ell(H)=n$ such that $G\subseteq H$. 
Then $H$ acts on a {\em tree} $\Gamma$ without inversion (of edges)
with $2$ orbits on $V(\Gamma)$ and $2$ orbits on  $E(\Gamma)$,
where $V(\Gamma)$ is the set of {\em vertices} of $\Gamma$, and  
$E(\Gamma)$ is the set of (oriented) {\em edges} of $\Gamma$ 
(cf. \cite[\S I.4.4]{ser:trees}).
Moreover, for $v\in V(\Gamma)$ its stabilizer $H_v$ is either free abelian
or an i.e.c.f. group of level $n-1$. 
For $e\in E(\Gamma)$ its stabilizer $H_e$ is infinite cyclic.

As $G$ is acting on $\Gamma$ without inversion of edges,
the fundamental theorem of Bass-Serre theory (cf. \cite[\S I.5.4, Thm.~13]{ser:trees})
implies that $G$ is isomorphic
to the fundamental group $\pi_1(\gG,\Lambda,\caT)$
of a graph of groups $\gG$ based on a connected graph $\Lambda$
and $\caT$ is a maximal subtree of $\Lambda$.
For simplicity we assume that $G=\pi_1(\gG,\Lambda,\caT)$.
Since $G$ is finitely generated, $E=E(\Lambda)\setminus E(\caT)$ must be finite.
Otherwise, $G$ would have an infinitely generated free group as a homomorphic image.
Similarly, as $G$ is finitely generated, there exists a finite set 
$\Omega\subseteq V(\Lambda)$, such that 
\begin{equation}
\label{eq:grogr}
G=\langle\,e,\,\gG_v\mid e\in E,\, v\in\Omega\,\rangle.
\end{equation}
Let $\caT_0=\spn_{\caT}(\Omega\cup\{\,o(e),t(e)\mid e\in E\,\})$ be the tree spanned by the
set $\Omega$ and the origins and termini of the edges in $E$, and let $\Lambda_0$ be the subgraph of $\Lambda$ satisfying $V(\Lambda_0)=V(\caT_0)$, and 
$E(\Lambda_0)=E(\caT_0)\sqcup E$. By construction, $\Lambda_0$ is a finite
connected graph,
and $\caT_0$ is a maximal subtree of $\Lambda_0$. Let 
$\gG^\prime=\gG\vert_{\Lambda_0}$ be the restriction of $\gG$ to $\Lambda_0$.
Then, by definition, one has a canonical homomorphism of groups
$i_\caT\colon \pi_1(\gG^\prime,\Lambda_0,\caT_0)\to
\pi_1(\gG,\Lambda,\caT)$. For $P_0\in V(\Lambda_0)$,
one has a commutative diagram
\begin{equation}
\label{eq:dia1}
\xymatrix{
\pi_1(\gG^\prime,\Lambda_0,P_0)\ar[r]^{i_{P_0}}\ar[d]&
\pi_1(\gG,\Lambda,P_0)\ar[d]\\
\pi_1(\gG^\prime,\Lambda_0,\caT_0)\ar[r]^{i_{\caT}}&
\pi_1(\gG,\Lambda,\caT)
}
\end{equation}
where the vertical maps are isomorphisms (cf. \cite[\S I.5.1, Prop.~20]{ser:trees}). As the canonical map is mapping generalized reduced words  
to generalized reduced words, $i_{P_0}$
is injective (cf. \cite[p.~50, Ex.~(c)]{ser:trees}), and hence $i_\caT$ is injective.
Thus, by \eqref{eq:grogr}, $i_\caT$ is an isomorphism. As a consequence, one concludes the following property which is slightly more precise than \cite[Thm.~6]{KM2}.

\begin{prop}
\label{prop:prin}
Let $G$ be a limit group of height $n\geq 1$. Then $G$ is isomorphic to the fundamental
group $\pi_1(\gG^\prime,\Lambda_0,\caT_0)$ of a graph of groups $\gG^\prime$
satisfying
\begin{itemize}
\item[(i)] $\Lambda_0$ is finite;
\item[(ii)] for all $v\in V(\Lambda_0)$, $\gG^\prime_v$ is finitely generated abelian
or a limit group of height at most $n-1$;
\item[(iii)] for all $e\in E(\Lambda_0)$, $\gG^\prime_e$ is infinite cyclic or trivial.
\end{itemize}
\end{prop}

\begin{proof}
By the previously mentioned argument, it suffices to show that $\gG^\prime_v$ is finitely generated for all $v\in V(\Lambda_0)$.
This follows by the argument used in \cite[Proof of Thm.~6, p.~567]{KM2} in connection
with \cite[Satz~5.8]{CRR} and the equivalent statement for HNN-extensions
(cf. \cite[Satz~6.3]{CRR}).
\end{proof}

%%%%

\subsection{Limit groups of small rank}
\label{ss:lsr} 
As limit groups are (fully) residually free, they must be torsion free.
From this property one concludes the following useful fact.

\begin{fact}
\label{L0}
Let $G$ be a limit group with $\rk_{\Q}(G)=1$. Then $G\simeq \Z$.
\end{fact}

\begin{proof}
Suppose that $G$ is non-abelian, i.e., there exist $a,b\in G$ with $[a,b]\neq 1$. 
Since $G$ is a limit group, there exists a free group $F$ of finite rank  and an  epimorphism $\phi\colon G\longrightarrow F$  satisfying $\phi([a,b])\neq 1$. 
As $\phi^{\ab}\colon G^{\ab}\to F^{\ab}$ is surjective, its induced map
$\phi^{\ab}_{\Q}\colon G^{\ab}\otimes_{\Z}\Q\to F^{\ab}\otimes_{\Z}\Q$ is also surjective.
Hence $1=\rk_{\Q}(G)\geq \rk_{\Q}(F)$,  and therefore $\rk_{\Q}(F)=1$. As 
$\rk_{\Q}(F)=d(F^{\ab})=d(F)=1$, $F$ must be cyclic, and, in particular, abelian.
Thus $\phi([a,b])=[\phi(a),\phi(b)]=1$, a contradiction. Hence $G$ is a finitely generated
free abelian group. 
In particular, $d(G)=\rk_{\Q}(G)=1$, and $G$ is cyclic. 
\end{proof}

The following property has been shown by
D.~Kochloukova in \cite{D}.

\begin{prop}[D.~Kochloukova]
\label{l4}
For a limit group $G$ its Euler characteristic $\chi(G)$ is non-positive. Moreover, 
$\chi(G)=0$ if, and only if,  $G$ is abelian. 
\end{prop}

Limit groups with minimum number of generators less or equal to $3$ are well known. In \cite{B} the following was shown.

\begin{thm}[B.~Fine, A.~Gaglione, A.G.~Myasnikov, G.~Rosenberger, D.~Spellman]
\label{t3}
Let $G$ be a limit group. Then
\begin{itemize}
\item[(a)] $d(G)=1$ if, and only if, $G$ is infinite cyclic;
\item[(b)] $d(G)=2$ if, and only if, $G$ is a free group of rank $2$ or a free abelian group of rank $2$;
\item[(c)] $d(G)=3$ if, and only if, $G$ is a free group of rank $3$, a free abelian group of rank $3$, or an extension of centralizers of a free group of rank $2$, i.e, $G$ has a
presentation
\begin{equation}
\label{eq:free3}   
G=\langle\, x_1,x_2,x_3\mid x_3^{-1}\,v\,x_3=v\,\rangle,
\end{equation} 
where $v=v(x_1,x_2)\in F(\{x_1,x_2\})$ is non-trivial. 
\end{itemize}
\end{thm}

By \cite[Thm.~3.1]{BK}, there are only few limit groups with the \IFP-property.

\begin{thm}[B.~Fine, O.G.~Kharlampovich, A.G.~Myasnikov, V.N.~Remeslennikov, G.~Rosenberger]
\label{T1}
Let $G$ be a limit group with the $\textbf{IF}$-property. Then $G$ is a free group or a cyclically pinched or a conjugacy pinched one-relator group. 
\end{thm}  

%%%%%%%%%%%%
%%%% Main Thm %%
%%%%%%%%%%%%
%%% May11th %%%%
\section{Generators of  normal subgroups of limit groups}
\label{nsb}
The following fact will turn out to be useful for our purpose.

\begin{Lemma}
\label{l2}
Let $G_1$ and  $G_2$ be  groups, and let $C=\langle c \rangle$ be an infinite cyclic group or the trivial group.
\begin{itemize}
\item[(a)] If $G=G_1\star_{C}G_2$ is a free product with amalgamation in $C$, then 
\begin{equation}
\label{eq:baslem1}
\rk_{\Q}(G)=\rk_{\Q}(G_1)+\rk_{\Q}(G_2)-\rho(G),
\end{equation}
where $\rho(G) \in \{0,1\}$. Moreover, if $C=1$, then $\rho(G)=0$. 
\item[(b)] If $G=\HNN_\phi(G_1,C,t)=\langle\, G_1,t\mid t\,c\, t^{-1}=\phi(c)\,\rangle$ is an HNN-extension with 
equalization in $C\subseteq G_1$, then 
\begin{equation}
\label{eq:baslem2}
\rk_{\Q}(G)=\rk_{\Q}(G_1)+\rho(G),
\end{equation} 
where $\rho(G)\in\{0,1\}$. One has an exact sequence 
\begin{equation}
\label{eq:baslem3}
\xymatrix{
C\otimes_{\Z}\Q\ar[r]^-{\alpha}& 
G_1^{\ab}\otimes_{\Z}\Q \ar[r]^-{\beta}&
G^{\ab}\otimes_{\Z}\Q  \ar[r]&
\Q\ar[r]&0,
}
\end{equation}
where $\alpha(c\otimes q)=(c\phi(c)^{-1}G_1^{'})\otimes q$, $q\in \Q$. Moreover, 
\begin{itemize}
\item[(1)] $\rho(G)=0$ if, and only if, $\alpha$ is injective;
\item[(2)] $\rho(G)=1$ if, and only if, $\alpha$ is the $0$-map. 
\end{itemize}
\end{itemize}
\end{Lemma}

\begin{proof}
(a) Let $G=G_1\star_{C}G_2$, and let $\iota_i\colon C\to G_i$, $i\in\{1,2\}$ denote the associated embeddings. 
The Mayer-Vietoris sequence associated to 
$-\otimes_{\Z}\Q$ specializes to an exact sequence
\begin{equation}
\label{eq:baslem4}
\xymatrix{
&H_1(C,\Q)\ar[r]^-{\alpha}& 
H_1(G_1,\Q)\oplus H_1(G_2,\Q)\ar[r]^-{\beta}&
H_1(G,\Q)\ar[d]^\gamma\\
0& H_0(G,\Q)\ar[l]&H_0(G_1,\Q)\oplus H_0(G_2,\Q)\ar[l]&H_0(C,\Q)\ar[l]
}
\end{equation} 
where $\alpha(c\otimes q)=((\iota_1(c)G_1^{\prime})\otimes q,-(\iota_2(c)G_2^{\prime})\otimes q)$, $q\in \Q$.
In particular, $\gamma=0$, and this yields (a).

\noindent
(b) In this case the Mayer-Vietoris sequence specializes to
\begin{equation}
\label{eq:baslem5}
\xymatrix{
&H_1(C,\Q)\ar[r]^-{\alpha}& 
H_1(G_1,\Q)\ar[r]^-{\beta}&
H_1(G,\Q)\ar[d]\\
0& H_0(G,\Q)\ar[l]&H_0(G_1,\Q)\ar[l]&H_0(C,\Q)\ar[l]_-{\delta}
}
\end{equation} 
with $\alpha$ as described in (b) of the statement of the lemma.
In particular, $\delta=0$ which yields \eqref{eq:baslem3}, and thus also \eqref{eq:baslem2}. The final
remarks (1) and (2) follow from the fact that
$\dim(\image(\alpha))\in\{0,1\}$, and that $\dim(\image(\alpha))=1$ if, and only if, $\alpha$ is injective.
\end{proof}

From Lemma \ref{l2} one concludes the following. 

\begin{thm}
\label{thm:A}
Let  $G$ be a non-abelian limit group, and let $U$ be a normal subgroup  of  $G$ of  prime index $p$. 
Then $\rk_{\Q}(U)> \rk_{\Q}(G)$.
\end{thm}

\begin{proof}
We proceed by induction on $n=\het(G)$ (cf. \S\ref{ss:hgt}).
If $n=0$, then $G$ is a finitely generated free group satisfying $d(G)\geq 2$, and the
Nielsen-Schreier theorem  yields 
\begin{equation}
\label{eq:mt1}
\rk_{\Q}(U)=d(U)=p\cdot(d(G)-1)+1> d(G)=\rk_{\Q}(G)
\end{equation} 
and hence the claim.
So assume that $G$ is a limit group of height $\het(G)=n\geq 1$, and that the claim holds
for all limit groups of height less or equal to $n-1$. By Proposition~\ref{prop:prin},
$G$ is isomorphic to the fundamental group $\pi_1(\Upsilon,\Lambda,\caT)$
of a graph of groups $\Upsilon$ based on a finite connected graph $\Lambda$ which edge groups are either infinite cyclic or trivial, and which vertex groups are either limit group of height at most $n-1$ or free abelian groups. 
Applying induction on $s(\Lambda)=|V(\Lambda)|+|E(\Lambda)|$ it suffices to consider
the following two cases:

\begin{itemize}
\item[(I)] $G=G_1 \star_{C} G_2$ and $G_i$ is either a limit group of height
at most $n-1$ or abelian, and $C$ is either infinite cyclic or trivial, $i\in\{1,2\}$;
\item[(II)] $G=\HNN_\phi(G_1,C,t)$ where $G_1$ is either a limit group of height
at most $n-1$ or abelian, and $C$ is either infinite cyclic or trivial.
\end{itemize}

\noindent
{\bf Case I:} Let  $G=G_1\star_CG_2$. If  $C$ is non-trivial, then either $G_1$ or $G_2$ is non-abelian. Otherwise, by a result of I.~Chiswell, one would conclude that
$\chi(G)= \chi(G_1)+\chi(G_2)-\chi(C)=0$ and $G$ must be abelian
(cf. Prop.~\ref{l4}) which was excluded by hypothesis.
The group $G$ acts naturally on a tree $T$ without inversion of edges with two orbits 
$V_1$ and $V_2$ on $V(T)$ and two orbits $E_1$ and $E_2$ on $E(T)$
satisfying $\bar{E_1}=E_2$. We may assume that for $e\in E_1$ one has
$o(e)\in V_1$ and $t(e)\in V_2$. By hypothesis,
\begin{equation}
\label{eq:mt2}
|G:U\,G_i|\in\{1,p\}\qquad\text{and}\qquad |G:U\,C|\in\{1,p\},
\end{equation}
and one may distinguish the following cases:
\begin{itemize}
\item[(I.1)] $|G:U\,C|=1$, i.e., $U$ has one orbit on $E_1$ and $E_2$;
\item[(I.2)] $|G:U\,C|=p$, i.e., $U$ has $p$ orbits on $E_1$ and $E_2$.
\end{itemize}
{\bf Case I.1:}  
The hypothesis implies that $G=U\,G_1=U\,G_2$, and 
$U$ has one orbit on $V_1$ and one orbit on $V_2$.
Thus, by the fundamental theorem of Bass-Serre theory 
(cf. \cite[\S I.5.4, Thm.~13]{ser:trees}),
one has that $U\simeq (U\cap G_1)\star_{(U\cap C)} (U\cap G_2)$.
The hypothesis implies also that $|C:C\cap U|=p$, i.e., $C\not=1$.
Hence without loss of generality we may assume that $G_1$ is non-abelian, and,
by induction, 
$\rk_{\Q}(U\cap G_1)\geq 1+\rk_{\Q}(G_1)$. If $G_2$ is also non-abelian, then, 
by induction, $\rk_{\Q}(U\cap G_2)\geq\rk_{\Q}(G_2)+1$.
Hence,
by applying Lemma \ref{l2}(a), one concludes that
\begin{equation}
\label{eq:mt3}
\rk_{\Q}(U)\geq \rk_{\Q}(U\cap G_1)+ \rk_{\Q}(U\cap G_2)-1
>\rk_{\Q}(G_1)+\rk_{\Q}(G_2)
\geq \rk_{\Q}(G).
\end{equation}
On the other hand, if $G_2$ is abelian,
$C$ is a direct factor in $G_2$, i.e., $G_2\simeq \Z\times B$, where $B$ is a free abelian group of rank $d(G_2)-1$. Thus the map $\alpha$ in \eqref{eq:baslem4} is injective,
and one has $\rk_{\Q}(G)=\rk_{\Q}(G_1)+\rk_{\Q}(G_2)-1$.
Moreover, $\rk_{\Q}(U\cap G_2)=\rk_{\Q}(G_2)$
and Lemma~\ref{l2}(a) yields
\begin{equation}
\label{eq:mt4}
\rk_{\Q}(U)\geq \rk_{\Q}((U\cap G_1)+ \rk_{\Q}(U\cap G_2)-1\geq 
\rk_{\Q}(G_1)+rk_{\Q}(G_2)>\rk_{\Q}(G).
\end{equation}
{\bf Case I.2:}  
In this case $U$ has $p$ orbits on $E_1$ and $p$ orbits on $E_2$.
Moreover, $|G:UG_i|\in\{1,p\}$. Hence, by \eqref{eq:mt2}, it suffices to consider the following three cases
\begin{itemize}
\item[{\bf (a)}] $|G:UG_1|=|G:UG_2|=p$;
\item[{\bf (b)}] $|G:UG_1|=p$ and $G=UG_2$;
\item[{\bf (c)}] $G=UG_1=UG_2$.
\end{itemize}
Let $\Lambda=G\dbs T$ and $\tLambda= U\dbs T$ be the quotient graphs of $T$ modulo the left $G$- and $U$-action, respectively.
In particular, $\Lambda=(\{v,w\},\{\be,\beb\})$ is a line segment of length~$1$
and thus a tree, and $\tLambda$ is connected. Moreover, 
one has a surjective homomorphism of graphs $\pi\colon \tLambda\to\Lambda$.

\noindent
{\bf Case (a):}  
By hypothesis, $UG_1=UG_2=U$, i.e., $G/U$ acts regularly on the
vertex fibres and edge fibres of $\pi$.
For $\bof\in E(\tLambda)$ one has either $o(\bof)\in\pi^{-1}(\{v\})$
or $t(\bof)\in\pi^{-1}(\{v\})$. If $\bof_1$ and $\bof_2$ satisfy the first condition,
and $\tv=o(\bof_1)=o(\bof_2)$, then  $\tg.\bof_1=\bof_2$ for 
$\tg\in\stab_{G/U}(\tv)=1$. Hence $\bof_1=\bof_2$. In the latter case the same argument applies, and this shows that $\pi$ is a fibration, i.e., 
for every $z\in V(\tLambda)$ the map
$\pi_z\colon \st_{\tLambda}(z)\to\st_{\Lambda}(\pi(z))$
is a bijection. As $\Lambda$ is a tree and hence simply-connected, this implies
that $\pi$ is a bijection, a contradiction, showing that Case~(a) is impossible. 

\noindent
{\bf Case (b):} We may assume that $G_1=\stab_G(v)$ and
$G_2=\stab_G(w)$. Hence, by hypothesis,
$|\pi^{-1}(\{v\})|=p$ and $|\pi^{-1}(\{w\})|=1$.
Let $\tw\in V(\tLambda)$ with $\pi(\tw)=w$.
Then $E(\tLambda)=\st_{\tLambda}(\tw)\sqcup \overline{\st_{\tLambda}(\tw)}$
is a star with $p$ geometric edges.

Put $U_2=U\cap G_2$. By hypothesis, $|G_2:U_2|=p$ and $G_1\subseteq U$. 
Choosing a set of representatives $\caR\subseteq G_2$ for $G_2/U_2$,
the Mayer-Vietoris sequence associated to $\Tor^U_\bullet(\argu,\Q)$ specializes to
\begin{equation}
\label{eq:mt5}
\xymatrix{
\coprod_{r\in\caR} H_1({}^rC,\Q)\ar[r]^-{\alpha}&
 \coprod_{r\in\caR} H_1({}^rG_1,\Q)\oplus H_1(U_2,\Q)\ar[r]
&H_1(U,\Q)\ar[r]&0}.
\end{equation}
This yields
\begin{equation}
\label{eq:mt6}
\rk_\Q(U)=p\cdot \rk_\Q(G_1)+\rk(U_2)-\delta,
\end{equation}
where $\delta=\dim(\image(\alpha))$. We distinguish two cases.

\noindent
{\bf (1)} $C=1$. So $\delta=0$. Then, by  \eqref{eq:mt5}, 
\begin{equation}
\label{eq:mt7}
\rk_{\Q} (U)=p\cdot \rk_{\Q}(G_1)+\rk_{\Q}(U_2).
\end{equation}
As $\rk_{\Q}(U_2)\geq \rk_{\Q}(G_2)$ with equality in case
that $G_2$ is abelian, one concludes that
$\rk_{\Q}(U)= p\cdot\rk_{\Q}(G_1)+\rk_{\Q}(U_2)
>\rk_{\Q}(G_1)+\rk_{\Q}(G_2)\geq \rk_{\Q}(G)$.

\noindent
{\bf (2)} $C\not=1$. Then, by \eqref{eq:mt6}, 
\begin{gather}
\rk_{\Q}(U)\geq p\cdot (\rk_{\Q}(G_1)-1)+\rk_{\Q}(U_2)
\geq 2\cdot (\rk_{\Q}(G_1)-1)+\rk_{\Q}(U_2).\label{eq:mt8}\\
\intertext{As $d(G_1)=1$ implies that $G=\Z\star_{\Z} G_2\simeq G_2$,
which was excluded by (I), we may assume that $d(G_1)\geq 2$.
Hence}
\rk_{\Q}(U) \geq\rk_{\Q}(G_1)+\rk_{\Q}(U_2).\label{eq:mt9} 
\end{gather}
If $G_2$ is non-abelian, then, by induction,  
$\rk_{\Q}(U_2)\geq 1+\rk_{\Q}(G_2)$. Hence 
$\rk_{\Q}(U)\geq1+\rk_{\Q}(G_1)+\rk_{\Q}(G_2)> \rk_{\Q}(G)$.
If $G_2$ is abelian, then $G=G_1\star_\Z (\Z\times B)$ for some free abelian group $B$ of  rank  $d(G_2)-1$. In particular,  $\rk_{\Q}(G)=\rk_{\Q}(G_1)+d(B)$.
As $\rk_{\Q}(U_2)=d(B)+1$, \eqref{eq:mt9} yields the claim also in this case.

\noindent
{\bf Case (c):}
By hypothesis, $\tLambda$ is a graph with two vertices $v$ and $w$ and
$p$ geometric edges. We may assume that $G_1=\stab_G(v)$, $G_2=\stab_G(w)$
and put $U_1=U\cap G_1$, $U_2=U\cap G_2$. By the same argument used in
Case (b), one obtains an exact sequence
\begin{equation}
\label{eq:casc1}
\xymatrix{
\coprod_{r\in\caR} H_1({}^rC,\Q)\ar[r]^-\beta&
H_1(U_1,\Q)\oplus H_1(U_2,\Q)\ar[r]& H_1(U,\Q)\ar[d]\\
 &0&\ar[l]\Q^{p-1}\\
}
\end{equation}
where $\caR\subset G$ is a set of representatives of $G/U$. Again we distinguish two cases.
 
\noindent
{\bf (1)} $C=1$. Then $\beta$ is the $0$-map, and as $\rk_{\Q}(U_i)\geq \rk_{\Q}(G_i)$ for $i\in\{1,2\}$, one has
\begin{equation}
\label{eq:casc2}
\rk_{\Q}(U)=\rk_{\Q}(U_1)+\rk_{\Q}(U_2)+(p-1)
\geq \rk_{\Q}(G_1)+\rk_{\Q}(G_2)+1>\rk_{\Q}(G).
\end{equation}
{\bf (2)} $C\neq 1$.
Then, by \eqref{eq:casc1},
\begin{equation}
\label{eq:casc3}
\rk_{\Q}(U)\geq \rk_{\Q}(U\cap G_1)+\rk_{\Q}(U\cap G_2)-1.
\end{equation} 
If $G_1$ and  $G_2$ are non-abelian, then,
by  induction, $\rk_{\Q}(U\cap G_1)\geq 1+\rk_{\Q}(G_1)$ and  
$\rk_{\Q}(U\cap G_2)>\rk_{\Q}(G_2)$.
Hence
\begin{equation}
\label{eq:casc4}
\rk_{\Q}(U)>\rk_{\Q}(G_1)+1+rk_{\Q}(G_2)-1=\rk_{\Q}(G_1)+\rk_{\Q}(G_2)\geq 
\rk_{\Q}(G).
\end{equation}
In case that one of the groups $G_1$, $G_2$ is abelian,
not both of them can be abelian. Otherwise, the Euler characteristic
$\chi(G)=\chi(G_1)+\chi(G_2)-\chi(C)$ must equal $0$,
and $G$ must be  abelian, which was excluded by hypothesis
(cf. Prop.~\ref{l4}).
So without loss of generality we may assume that
$G_1$ is a non-abelian, and that $G_2$ is abelian.
Then $G_2\simeq \Z\times B$, where $B$ is a free abelian group of rank  
$\rk_{\Q}(G_2)-1$, and $G\simeq G_1\star_{\Z}(\Z\times B)$.
Hence  $\rk_{\Q}(G)=\rk_{\Q}(G_1)+\rk_{\Q}(G_2)-1$. 
Furthermore, $\rk_{\Q}(U\cap G_2)=\rk_{\Q}(G_2)$, and,
by induction, $\rk_{\Q}(U\cap G_1)\geq 1+\rk_{\Q}(G_1)$.
Thus, by \eqref{eq:casc3},
\begin{equation}
\rk_{\Q}(U)
\geq \rk_{\Q}(G_1)+1+\rk_{\Q}(G_2)-1=\rk_{\Q}(G_1)+\rk_{\Q}(G_2)> \rk_{\Q}(G).
\end{equation}
{\bf Case II:} Let  $G=\HNN_\phi(G_1,C,t)=
\langle\, G_1,t\mid\,t\,c\,t^{-1}=\phi(c)\,\rangle$ be an HNN-extension
with   $C=\langle c \rangle$. 
By \eqref{eq:baslem2}, one has
\begin{equation}
\label{eq:hnn1}
\rk_{\Q}(G_1)\leq \rk_{\Q}(G)\leq \rk_{\Q}(G_1)+1.
\end{equation} 
If $C=1$, then $G=G_1\star  <t>$ is isomorphic to a free product.
Hence the claim follows already from Case I. So we may assume that $C\not=1$.
Note that $G_1$ must be non-abelian.
Otherwise, one has $\chi(G)=\chi(G_1)-\chi(C)=0$, 
and $G$ must be abelian (cf. Prop.~\ref{l4}), a contradiction. 

As in Case I, the group $G$ has a natural vertex transitive action on a tree $T$
with vertex stabilizer isomorphic to $G_1$, and edge stabilizer isomorphic to $C$.
In particular,
\begin{equation}
\label{eq:hnn2}
|G:UG_1|\in\{1,p\},\qquad\text{and}\qquad |G:UC|\in\{1,p\}.
\end{equation}
Hence one may distinguish the following two cases:
\begin{itemize}
\item[(II.1)] $G=UC$;
\item[(II.2)] $|G:UC|=p$.
\end{itemize}
\noindent
{\bf Case II.1:} By hypothesis, $G=UG_1$. In particular, $U$ is acting vertex transitively on $T$, and has two orbits on the set of edges, i.e., $U\dbs T$ is a loop with one vertex.
Thus $U\simeq \HNN_\phi(U\cap G_1, C^p, t)=
\langle\,U\cap G_1,t\mid t\,(c^p)\,t^{-1}=\phi(c)^{p}\,\rangle$, and as in \eqref{eq:hnn1}, one concludes that
\begin{equation}
\label{eq:hnn3}
\rk_{\Q}(U\cap G_1)\leq \rk_{\Q}(U)\leq \rk_{\Q}(U\cap G_1)+1.
\end{equation}
Since $G_1$ is non-abelian and $|G_1:U\cap G_1|=p$, induction implies that  
$\rk_{\Q}(G_1\cap U)\geq \rk_{\Q}(G_1)+1$. Hence  
\begin{equation}
\label{eq:hnn4}
\rk_{\Q}(U)\geq \rk_{\Q}(G_1\cap U)\geq \rk_{\Q}(G_1)+1\geq \rk_{\Q}(G).
\end{equation}  
Suppose that $\rk_{\Q}(U)=\rk_{\Q}(G)$. Then 
$\rk_{\Q}(U)=\rk_{\Q}(G_1\cap U)=\rk_{\Q}(G_1)+1=\rk_{\Q}(G)$. 
In particular, by Lemma \ref{l2}(b), $\rho({U})=0$ and $\rho(G)=1$, i.e.,
$\alpha$ is the $0$-map, and $\alpha_1$ is injective.
However, as 
the map $\tr_1$
in the diagram
\begin{equation}
\label{eq:hnn5}
\xymatrix{
\ldots\ar[r]&H_1(C,\Q)\ar[r]^\alpha\ar[d]_{\tr_1}& H_1(G_1,Q)\ar[d]_{\tr_2}\ar[r]&
H_1(G,\Q)\ar[d]^{\tr_3}\ar[r]&\ldots\\
\ldots\ar[r]&H_1(C^p,\Q)\ar[r]^{\alpha_1}& H_1(U\cap G_1,Q)\ar[r]&H_1(U,\Q)\ar[r]&
\ldots}
\end{equation}
- which is given by the transfer -
is an isomorphism, this yields a contradiction. Thus $\rk_{\Q}(U)<\rk_{\Q}(G)$.

\noindent
{\bf Case II.2:} In this case $U$ has $2p$ orbits on $E(T)$, and again one may distinguish two cases:
\begin{itemize}
\item[{\bf (a)}] $U$ acts vertex transitively on $T$;
\item[{\bf (b)}] $U$ has $p$ orbits on $V(T)$.
\end{itemize}

\noindent
{\bf Case (a):}
In this case $U\dbs T$ is a bouquet of $p$ loops, 
$|G_1\colon U\cap G_1|=p$ and $C\subseteq U$.
The Mayer-Vietoris sequence
for $H_\bullet(\argu,\Q)$ yields a commutative and exact diagram
\begin{equation}
\label{eq:hnn6}
\xymatrix{
C\otimes_\Z\Q\ar[r]^{\alpha}\ar[d]_{\tr_1}&G_1^{\ab}\otimes_\Z\Q\ar[r]\ar[d]_{\tr_2}&
G^{\ab}\otimes_\Z\Q\ar[r]\ar[d]^{\tr_3}&\Q\ar[r]\ar[d]&0\\
\coprod_{i=1}^p {}^{g_i}C\otimes_\Z\Q\ar[r]^{\alpha_1}&(U\cap G_1)^{\ab}\otimes_\Z\Q\ar[r]&
U^{\ab}\otimes_\Z\Q\ar[r]&\Q^p\ar[r]&0,
}
\end{equation}
where $\{\,g_1,\ldots,g_p\}\subseteq G_1$ is a set of representatives for $G/U$,
and $\tr_{2/3}$ is induced by the transfer, and for $c\in C$, $q\in \Q$, one has
$\tr_1(c\otimes q)= \sum_{i=1}^p {}^{g_i} c\otimes q$. Moreover, 
\begin{equation}
\label{eq:hnn7}
\alpha_1((^{g_i}c)\otimes q)=\big(g_i\,c\,\phi(c)^{-1}g_i^{-1}(U\cap G_1)^\prime\big)\otimes q, \qquad\text{for $c\in C$.}
\end{equation} 
Thus, one has
\begin{equation}
\label{eq:hnn8}
\rk_{\Q}(U)+\dim_{\Q}(\image(\alpha_1))=\rk_{\Q}(U\cap G_1)+p,
\end{equation}
and induction implies that  $\rk_{\Q}(U\cap G_1)\geq \rk_{\Q}(G_1)+1$. 
Hence 
\begin{equation}
\label{eq:hnn9}
\rk_{\Q}(U)\geq \rk_{\Q}(U\cap G_1)\geq \rk_{\Q}(G_1)+1\geq \rk_{\Q}(G).
\end{equation}
Suppose that $\rk_{\Q}(U)=\rk_{\Q}(G)$.
Then equality holds throughout \eqref{eq:hnn9}. 
In particular, \eqref{eq:hnn8} implies that
$\dim_{\Q}(\image(\alpha_1))=p $, i.e., $\alpha_1$ is injective.
Hence, as $\tr_1$ is injective, $\alpha_1\circ\tr_1$ is injective.
From Lemma \ref{l2}(b) one concludes that $\rho(G)=1$ and $\alpha=0$, 
a contraction. Thus $\rk_{\Q}(U)>\rk_{\Q}(G)$ must hold.
 
\noindent
{\bf Case (b):}    
In this case one has $G_1\subseteq U$ and $C\subseteq U$.
As $G/U$ is acting vertex transitively on 
$\Lambda=U\dbs T$, $\Lambda$ must be $k$-regular.
Hence $|E(\Lambda)|=k\cdot |V(\Lambda)|$, forcing $k=2$.
So $\Lambda$ is a $2$-regular connected graph, and thus
a circuit with $p$ vertices.

Let $\{\,g_1,\ldots,g_p\}\subseteq G$ be a set of representatives for $G/U$.  
Then the Mayer-Vietoris sequence
for $H_\bullet(\argu,\Q)$ yields a commutative and exact diagram
\begin{equation}
\label{eq:hnn10}
\xymatrix{
C\otimes_\Z\Q\ar[r]^{\alpha}\ar[d]_{\tr_1}&G_1^{\ab}\otimes_\Z\Q\ar[r]\ar[d]_{\tr_2}&
G^{\ab}\otimes_\Z\Q\ar[r]\ar[d]^{\tr_3}&\Q\ar[r]\ar[d]&0\\
\coprod_{i=1}^p {}^{g_i}C\otimes_\Z\Q\ar[r]^{\alpha_1}&
\coprod_{i=1}^p{}^{g_i}G_1^{\ab}\otimes_\Z\Q\ar[r]&
U^{\ab}\otimes_\Z\Q\ar[r]&\Q\ar[r]&0,
}
\end{equation}
where $\tr_{1/2}$ are the diagonal maps and
$\tr_3$ is the transfer.
Hence
\begin{equation}
\label{eq:hnn11} 
\rk_{\Q}(U)=1+p\cdot (\rk_{\Q}(G_1)-1)+\dim_{\Q}(\kernel({\alpha_1})).
\end{equation}
Therefore, as $(p-1)\cdot(\rk_{\Q}(G_1)-1)\geq 1$,
\begin{equation}
\label{eq:hnn12}
\rk_{\Q}(U)\geq  \rk_{\Q}(G_1)+1\geq \rk_{\Q}(G).
\end{equation}  
Suppose that $\rk_{\Q}(U)=\rk_{\Q}(G)$. Then equality has to hold throughout
\eqref{eq:hnn12}, and 
$\dim_{\Q}(\kernel({\alpha_1}))=0$, $p=2$, $\rk_{\Q}(G_1)=2$.
In particular, $\alpha_1$ is injective, and by Lemma~\ref{l2}(b),
$\rho(G)=1$, i.e., $\alpha=0$. Hence, as $\tr_1$ is injective,
one obtains a contradiction as in Case (a) showing that $\rk_{\Q}(U)>\rk_{\Q}(G)$.
\end{proof}

Theorem~\ref{thm:A} has the following consequence.

\begin{cor}
\label{cor1}
Let $G$ be a non-abelian limit group such that $G^{\ab}$ is torsion free and that $d(G)=d(G^{\ab})$. Then, if $U$ is a normal subgroup of $G$ of prime index $p$,  one has that $d(U)>d(G)$.
\end{cor}

\begin{proof}
By Theorem~\ref{thm:A}, one has 
$d(U)\geq \rk_{\Q}(U)>\rk_{\Q}(G)=d(G^{\ab})=d(G)$. 
\end{proof} 

Obviously, if $G$ is a free group, an abelian free group
or an iterated extension of centralizers group of a free group, then 
$\rk_{\Q}(G)=d(G)$. In particular, in connection with Theorem~\ref{t3}
one concludes the following.

\begin{cor}
\label{cor2}
Let $G$ be a non-abelian limit group satisfying $d(G)\leq 3$,
and let $U$ be a normal subgroup of $G$ of prime index $p$.
Then $d(U)\geq \rk_{\Q}(U)>\rk_{\Q}(G)=d(G)$.
\end{cor}

%%%%%%%%%%%%%%%
%%% 7/7/2016 %%%%%%%
%%%%%%%%%%%%%%%
\section{Limit Groups  with the \IFP-property}
\label{homfin}
For a limit group $G$ it is not necessarily true that $d(G) = d (G^{\ab})$.
The following lemma shows that there exists a class of limit groups 
containing groups $G$ satisfying $d(G) \neq d (G^{\ab})$ 
(cf. Remark~\ref{rem:notab})
for which Question* has an affirmative answer.

\begin{Lemma}
\label{r2}
Let $G=G_1\star_C G_2$ be a non-abelian limit group, 
where $G_1$ and $G_2$ are free groups of finite rank $r(G_1)$ and $r(G_2)$, respectively,
and let $C=\langle c \rangle$ be an infinite cyclic group or trivial. 
Then, if $U$ is a normal subgroup of prime index  $p$ in $G$, one has 
$\rk_{\Q}(U)> d(G)$. In particular, $d(U)>d(G)$.
\end{Lemma}

\begin{proof}
If $C=1$, $G$ is a finitely generated free group, and there is nothing to prove.
So from now on we may assume that $C\neq 1$.
Then, as $G$ is non-abelian, either $G_1$ or $G_2$ must be
non-abelian. Otherwise, $\chi(G)=\chi(G_1)+\chi(G_2)-\chi(C)=0$, and by Lemma~\ref{l4}, $G$ must be abelian.
Let $T$ be the tree on which $G$ acts naturally.
Then, as in the proof of Theorem~\ref{thm:A}, Case I,
one may distinguish three cases:
\begin{itemize}
\item[(1)] $G=UC$, i.e., $U$ has one orbits on $E_1$ and $E_2$;
\item[(2)] $|G:UC|=p$, i.e., $U$ has $p$ orbits on $E_1$ and $E_2$;
\begin{itemize}
\item[(b)] $|G:UG_1|=p$ and $G=UG_2$;
\item[(c)] $G=UG_1=UG_2$.
\end{itemize}
\end{itemize}

\noindent
{\bf Case 1:}
By hypothesis, $G=UG_1$ and $G=UG_2$ and therefore,
\begin{equation}
\label{eq:lA1}
U\simeq (U\cap G_1)\star_{U\cap C} (U\cap G_2).
\end{equation}
In analogy to Case~I.1 of the proof of Theorem~\ref{thm:A},
we may choose $G_1$ to be non-abelian.
If $G_2$ is abelian, then $G_2$ must be cyclic, i.e., $d(G_2)=1$. If $d(G_1)=2$, then
$d(G)\leq 3$, and the claim follows from Corollary~\ref{cor2}.
Thus, we may assume that $d(G_1)\geq 3$.
Again, by the Nielsen-Schreier theorem (and as $d(G_1)\geq 3$), 
one has that $d(U\cap G_1)> 1+d(G_1)$.    
As $G_2$ is abelian, $d(U\cap G_2)=d(G_2)$.
Therefore, by Lemma \ref{l2}(a), one has
\begin{equation}
\label{eq:lA2}
\rk_{\Q}(U)\geq d(U\cap G_1)+ d(U\cap G_2)-1> d(G_1)+d(G_2)\geq d(G).
\end{equation}
Thus, we may assume that $G_2$ is non-abelian.
From the Nielsen-Schreier theorem one concludes that $d(U\cap G_1)\geq 1+d(G_1)$ and $d(U\cap G_2)\geq 1+d(G_2)$.  Therefore, by Lemma \ref{l2}(a),
\begin{equation}
\label{eq:lA3}
\rk_{\Q}(U)\geq d(U\cap G_1)+ d(U\cap G_2)-1\geq d(G_1)+d(G_2)+1 > d(G).
\end{equation}
\noindent
{\bf Case 2(b):}
By \eqref{eq:mt8}, one has that
\begin{equation}
\label{eq:lA4}
\rk_{\Q}(U)\geq p\cdot (d(G_1)-1)+d(G_2\cap U).
\end{equation}
If $G_1$ is  abelian, then $d(G_1)=1$.  In this case $G_2$ is  non-abelian.
If $d(G_2)=2$,  then $d(G)\leq 3$ and the claim follows from Corollary~\ref{cor2}.
So we may assume that $d(G_2)\geq 3$. In particular, 
the Nielsen-Schreier theorem and \eqref{eq:lA4} imply that
$\rk_{\Q}(U)\geq d(G_2\cap U)>1+d(G_2)\geq d(G)$.

So from now on we may assume that $G_1$ is non-abelian, i.e., $d(G_1)\geq 2$. 
If $G_2$ is abelian, then $d(U\cap  G_2)=d(G_2)=1$. So if $d(G_1)=2$, then 
the claim follows by Corollary~\ref{cor2}. 
Thus we may assume that $d(G_1)\geq 3$. 
Then, by \eqref{eq:lA4} and the Nielsen-Schreier theorem, one concludes that
\begin{equation}
\label{eq:lA5} 
\rk_{\Q}(U)\geq p\cdot (d(G_1)-1)+1\geq d(G_1)+2\cdot (p-1)>d(G_1)+1\geq d(G).
\end{equation}
So we may assume that $G_2$ is non-abelian.
In this case the Nielsen-Schreier theorem implies that 
$d(U\cap  G_2)\geq 1+d(G_2)$. Then, by \eqref{eq:lA4}, one has
\begin{equation}
\label{eq:lA6}
\begin{aligned}
\rk_{\Q}(U)&\geq
p\cdot d(G_1)-p+d(G_2)+1\geq
d(G_1)+d(G_2)+(p-1)\\
&\geq d(G_1)+d(G_2)+1>d(G).
\end{aligned}
\end{equation} 
\noindent
{\bf Case 2(c):}
By \eqref{eq:casc3}, one has
\begin{equation}
\label{eq:lA7}
\rk_{\Q}(U)\geq d(U\cap G_1)+d(U\cap G_2)-1.
\end{equation}
Then, the proof of Case 1 can be transferred verbatim in order to show that
the claim holds if one of the groups $G_i$, $i\in\{1,2\}$ is abelian.
So we may assume that $G_1$ and $G_2$ are non-abelian.
Then the Nielsen-Schreier theorem and the same argument which was used 
in order to prove \eqref{eq:lA3} implies that $\rk_{\Q}(U)> d(G)$.
\end{proof}

For HNN-extensions one has the following.

\begin{Lemma}
\label{r1}
Let $G=\HNN_\phi(G_1,C,t)=\langle\, G_1,t\mid t\,c\,t^{-1}=\phi(c)\,\rangle$ be a non-abelian
limit group, where $G_1$ is a free group of finite rank $r$ and $C$ is an infinite cyclic group or trivial.
Let $U$ be a normal subgroup of $G$ of prime index $p$.
Then  $\rk_{\Q}(U)>d(G)$, and, in particular,  $d(U)>d(G)$.
\end{Lemma}

\begin{proof}
If $C$ is trivial, then $G=G_1\star <t>$ is a free group of rank $r+1$,
and there is nothing to prove. Moreover, if $r=1$, then $\chi(G)=0$, and $G$ must be abelian
(cf. Prop.~\ref{l4}),
which was excluded by hypothesis. Hence $r\geq 2$.
Let $T$ be the tree on which $G$ acts naturally.
Then one may distinguish three cases:
\begin{itemize}
\item[(1)] $G=UC$;
\item[(2)] $|G:UC|=p$;
\begin{itemize}
\item[(a)] $U$ acts vertex transitiely on $T$;
\item[(b)] $U$ has $p$ orbits on $V(T)$.
\end{itemize}
\end{itemize}
 
\noindent
{\bf Case 1:}
In this case one has $U\simeq \HNN_{\phi}(U\cap G_1,C^p,t)$ (cf. Pf. of Theorem~\ref{thm:A}, Case II.1).
If $r=2$, then $d(G)\leq 3$, and the claim follows from Corollary~\ref{cor2}.
Hence we may assume that $r\geq 3$.
In this case the Nielsen-Schreier theorem implies that that $d(G_1\cap U)> r+1$, and by
\eqref{eq:hnn3}, one concludes that
\begin{equation}
\label{eq:HNN1}
\rk_{\Q}(U)\geq d(G_1\cap U)> r+1\geq d(G).
\end{equation}
{\bf Case 2(a):}
As in Case 1, we may assume that $r\geq 3$.
Then, as $|G_1:G_1\cap U|=p$ and $G_1$ is  a non-abelian free group, the Nielsen-Schreier theorem implies that $d(U\cap G_1)> r+1.$ 
Hence, by \eqref{eq:hnn4}, 
\begin{equation}
\label{eq:HNN2}
\rk_{\Q}(U)\geq d(U\cap G_1)> r+1\geq d(G).
\end{equation}
{\bf Case 2(b)}
As in Case 1, we may assume that $r\geq 3$.
By \eqref{eq:hnn11} and the fact that
$(p-1)\cdot (d(G_1)-1)>1$, one concludes that
\begin{equation}
\label{eq:HNN3}
\rk_{\Q}(U)\geq 1+p\cdot (d(G_1)-1)>d(G_1)+1\geq d(G)
\end{equation}
completing the proof of the lemma.
\end{proof}

As a consequence one concludes the following.

\begin{thm}
\label{thm:1rel}
Let $G$ be a limit group with the \textbf{IF}-property, 
and let $U$ be a normal subgroup of $G$ of prime index prime $p$.
Then $d(U)> d(G)$. 
\end{thm}
 
\begin{proof}
By Theorem~\ref{T1},
$G$ is either a free group, or a cyclically pinched one relator group,
or a conjugacy pinched one relator group.
In the first case there is nothing to prove.
In the second case $G\simeq G_1\star_C G_2$ with $G_1$ and $G_2$ free groups,
and $C$ being an infinite cyclic group being generated by a cyclically reduced word $w\in G_1$
and $C$ is a maximal cyclic subgroup in either $G_1$ or $G_2$.
Hence in this case, Lemma~\ref{r2} yields the claim.
If $G$ is a conjugacy  pinched one-relator group, then
$G\simeq \HNN_\phi(G_1,C,t)$ with $G_1$ a free group,
$C$ and infinite cyclic subgroup being generated by a cyclically reduced word $w\in G_1$,
and either $C$ or $\phi(C)$ is a maximal cyclic subgroup of $G_1$.
Then, by Lemma \ref{r1}, one has that $d(U)>d(G)$ completing the proof.
\end{proof}

\begin{rem}
\label{rem:notab}
It should be mentioned that there exist limit groups $G$
satisfying $d(G^{\ab})\lneq d(G)$.
Let $G_1=G_2=F_2$ be the free group of rank two, and let 
$C_1=C_2=\langle w \rangle$, where $w=a^2ba^{-1}b^{-1}$.
Then $w$ is cyclically reduce, and $C_i$ is a maximal cyclic subgroup of $G_i$ for $i=1,2$,
but $C_i$ is not a free factor. 
The corresponding double $G=G_1\star_{C_1=C_2}G_2$ is a non-abelian
limit group and has abelianization $G^{ab}$ isomorph to $\Z^3$, i.e., $d(G^{ab})=3$. 
By construction, there exists a surjective homomorphism
$\beta\colon G\to (F_2/\langle w^{F_2}\rangle)\star (F_2/\langle w^{F_2}\rangle)$,
and by hypothesis and Grushko's theorem,
$d(F_2/\langle w^{F_2}\rangle\star F_2/\langle w^{F_2}\rangle)=4$. Hence $d(G)=4$.
Note that by Lemma~\ref{r2}, $\rk_{\Q}(U)>4$ for any normal subgroup $U$ of $G$ of prime index $p$.
\end{rem}

\begin{rem}
\label{rem:tflim}
The group $G=\langle a,b\rangle_{a^2b^2=c^2d^2}\langle c,d\rangle$
is the fundamental group of a closed non-orientable surface, and it is known that it is a limit group
(cf. \cite[\S3.1]{Guirardel}).
However, $G^{\ab}\simeq \Z^3\times C_2$, where $C_2=\Z/2Z$.
\end{rem}

%\bibliography{limitbib}

\begin{thebibliography}{10}

\bibitem{AB}
E.~Alibegovi{{\'c}} and M.~Bestvina, \emph{Limit groups are {$\rm CAT(0)$}}, J.
  London Math. Soc. (2) \textbf{74} (2006), no.~1, 259--272. \MR{2254564}

\bibitem{Ba}
B.~Baumslag, \emph{Residually free groups}, Proc. London Math. Soc. (3)
  \textbf{17} (1967), 402--418. \MR{0215903}

\bibitem{BF}
M.~Bestvina and M.~Feighn, \emph{Notes on {S}ela's work: limit groups and
  {M}akanin-{R}azborov diagrams}, Geometric and cohomological methods in group
  theory, London Math. Soc. Lecture Note Ser., vol. 358, Cambridge Univ. Press,
  Cambridge, 2009, pp.~1--29. \MR{2605174}

\bibitem{CRR}
T.~Camps, V.~gro{\ss}e Rebel, and G.~Rosenberger, \emph{Einf{\"u}hrung in die
  kombinatorische und die geometrische {G}ruppentheorie}, Berliner Studienreihe
  zur Mathematik [Berlin Study Series on Mathematics], vol.~19, Heldermann
  Verlag, Lemgo, 2008. \MR{2378619}

\bibitem{CG}
C.~Champetier and V.~Guirardel, \emph{Limit groups as limits of free groups},
  Israel J. Math. \textbf{146} (2005), 1--75. \MR{2151593}

\bibitem{B}
B.~Fine, A.~M. Gaglione, A.~Myasnikov, G.~Rosenberger, and D.~Spellman, \emph{A
  classification of fully residually free groups of rank three or less}, J.
  Algebra \textbf{200} (1998), no.~2, 571--605. \MR{1610668}

\bibitem{BK}
B.~Fine, O.~G. Kharlampovich, A.~G. Myasnikov, V.~N. Remeslennikov, and
  G.~Rosenberger, \emph{On the surface group conjecture}, Sci. Ser. A Math.
  Sci. (N.S.) \textbf{15} (2007), 1--15. \MR{2367908}

\bibitem{Guirardel}
V.~Guirardel, \emph{Limit groups and groups acting freely on {$\mathbb
  R^n$}-trees}, Geom. Topol. \textbf{8} (2004), 1427--1470 (electronic).
  \MR{2119301}

\bibitem{Kapovich}
I.~Kapovich, \emph{Subgroup properties of fully residually free groups}, Trans.
  Amer. Math. Soc. \textbf{354} (2002), no.~1, 335--362 (electronic).
  \MR{1859278}

\bibitem{KMRS}
O.~G. Kharlampovich, A.~G. Myasnikov, V.~N. Remeslennikov, and D.~E. Serbin,
  \emph{Subgroups of fully residually free groups: algorithmic problems}, Group
  theory, statistics, and cryptography, Contemp. Math., vol. 360, Amer. Math.
  Soc., Providence, RI, 2004, pp.~63--101. \MR{2105437}

\bibitem{KM1}
O.~G. Kharlampovich and A.~M. Myasnikov, \emph{Irreducible affine varieties
  over a free group. {I}. {I}rreducibility of quadratic equations and
  {N}ullstellensatz}, J. Algebra \textbf{200} (1998), no.~2, 472--516.
  \MR{1610660}

\bibitem{KM2}
\bysame, \emph{Irreducible affine varieties over a free group. {II}. {S}ystems
  in triangular quasi-quadratic form and description of residually free
  groups}, J. Algebra \textbf{200} (1998), no.~2, 517--570. \MR{1610664}

\bibitem{D}
D.~H. Kochloukova, \emph{On subdirect products of type {${\rm FP}_m$} of limit
  groups}, J. Group Theory \textbf{13} (2010), no.~1, 1--19. \MR{2604842}

\bibitem{pav:pplim}
D.~H. Kochloukova and P.~A. Zalesskii, \emph{Subgroups and homology of
  extensions of centralizers of pro-{$p$} groups}, Math. Nachr. \textbf{288}
  (2015), no.~5-6, 604--618. \MR{3338916}

\bibitem{Paulin}
F.~Paulin, \emph{Sur la th{\'e}orie {\'e}l{\'e}mentaire des groupes libres
  (d'apr{\`e}s {S}ela)}, Ast{\'e}risque (2004), no.~294, ix, 363--402.
  \MR{2111650}

\bibitem{S1}
Z.~Sela, \emph{Diophantine geometry over groups. {I}. {M}akanin-{R}azborov
  diagrams}, Publ. Math. Inst. Hautes {\'E}tudes Sci. (2001), no.~93, 31--105.
  \MR{1863735}

\bibitem{S2}
\bysame, \emph{Diophantine geometry over groups. {II}. {C}ompletions, closures
  and formal solutions}, Israel J. Math. \textbf{134} (2003), 173--254.
  \MR{1972179}

\bibitem{ser:gal}
J-P. Serre, \emph{Galois cohomology}, english ed., Springer Monographs in
  Mathematics, Springer-Verlag, Berlin, 2002, Translated from the French by
  Patrick Ion and revised by the author. \MR{1867431}

\bibitem{ser:trees}
\bysame, \emph{Trees}, Springer Monographs in Mathematics, Springer-Verlag,
  Berlin, 2003, Translated from the French original by John Stillwell,
  Corrected 2nd printing of the 1980 English translation. \MR{1954121}

\end{thebibliography}
%\bibliographystyle{amsplain}
\end{document}